\documentclass[11pt]{amsart}

\topskip  \usepackage{auto-pst-pdf} % use with pstricks Use Tex & Dvi
\usepackage{pstricks,pst-node}
\usepackage{amsmath,amsthm}
\usepackage{amssymb}
\usepackage{color}
\usepackage{xspace}
\usepackage{tikz}

\newcommand*\circled[1]{\tikz[baseline=(char.base)]{
\node[draw,shape=circle,inner sep=1pt] (char) {#1};}}

\usepackage[colorlinks=true,linkcolor=green,filecolor=brown,citecolor=green]{hyperref}
\def\blue{\textcolor{blue} }
\def\white{\textcolor{white} }
\def\v{\vert}
\def\gf{generating function\xspace}

\def\gfNEW{\frac{1 - 5 x + (1 + x) \sqrt{1-4x}}{1 - 5 x + (1 - x) \sqrt{1-4x}}}

\begin{document}
\newtheorem{theorem}{Theorem}
\newtheorem{defn}[theorem]{Definition}
\newtheorem{lemma}[theorem]{Lemma}
\newtheorem{prop}[theorem]{Proposition}
\newtheorem{cor}[theorem]{Corollary}
\newtheorem{conjecture}[theorem]{Conjecture}

\title[Five subsets of permutations]{Five subsets of permutations enumerated as weak sorting permutations}
\author[D. Callan]{David Callan}
\address{Department of Statistics, University of Wisconsin, Madison, WI 53706}
\email{callan@stat.wisc.edu}
\author[T.~Mansour]{Toufik Mansour}
\address{Department of Mathematics, University of Haifa, 3498838 Haifa, Israel}
\email{tmansour@univ.haifa.ac.il}
\maketitle

\begin{abstract}
We show that the number of permutations of $\{1,2,\dots,n\}$ that avoid any 
one of five specific triples of 4-letter patterns is given by sequence A111279 in OEIS, 
which is known to count weak sorting permutations. By numerical evidence, there are no 
other (non-trivial) triples of 4-letter patterns giving rise to this sequence.  We make 
use of a variety of methods in proving our result, including recurrences, the kernel 
method, direct counting, and bijections.
\end{abstract}

\noindent{\emph{Keywords}}: pattern avoidance, Wilf-equivalence, kernel method, weak sorting permutations

\noindent{\emph{2010 Mathematics Subject Classification}}: 05A15, 05A05

\section{Introduction} \label{intro}

Let $\pi=\pi_1\pi_2\cdots\pi_n\in S_n$ and $\tau\in S_k$ be two permutations.  
We say that $\pi$ {\em contains} $\tau$ if there exists a subsequence $1\leq i_1<i_2<\cdots<i_k\leq n$ such that
$\pi_{i_1}\pi_{i_2}\cdots\pi_{i_k}$ is order-isomorphic to $\tau$; in such a context $\tau$ is usually called a {\em pattern}. We say that $\pi$ {\em avoids} $\tau$, or is {\em $\tau$-avoiding}, if no such subsequence exists. The set of all $\tau$-avoiding permutations in $S_n$ is denoted $S_n(\tau)$.
For an arbitrary finite collection of patterns $T$, we say that $\pi$ {\em avoids} $T$ if $\pi$
avoids every $\tau\in T$; the corresponding subset of $S_n$ is denoted $S_n(T)$. 
Two sets of patterns $T$ and $T'$ are said to be \emph{Wilf-equivalent} if their avoiders have the same 
counting sequence, that is, if $|S_n(T)|=|S_n(T')|$ for all $n \geq0$. 
In the context of pattern avoidance, a {\em symmetry class} refers to an orbit of the dihedral group of order eight generated by the  operations reverse, complement, and inverse acting entrywise on sets of patterns. 
Two pattern sets in the same symmetry class obviously have equinumerous avoiders, that is, are trivially Wilf-equivalent.

The weak sorting permutations are those that avoid 3241, 3421 and 4321 \cite{AAA}, counted by sequence \htmladdnormallink{A111279}{http://oeis.org/A111279} in \cite{oeis}.
We will show that there are precisely five symmetry classes of triples of 4-letter patterns counted as the weak sorting 
permutations. Representatives $\Pi_j, \ 1\le j \le 5$, of these five classes are listed in Theorem \ref{mainthm} below. 
The  weak sorting triple 3241, 3421, 4321 is in the same symmetry class as $\Pi_1$. 
(Our proof for $\Pi_1$ is different from that in \cite{AAA} for the weak sorting triple and is included because similar methods are used for $\Pi_2$ and $\Pi_3$.) 
A computer check of initial terms shows that no other symmetry 
class of triples of 4-letter patterns has this counting sequence.

\begin{theorem}[Main Theorem]\label{mainthm}
Define
\begin{align*}
\Pi_1=\{1234,1243,1342\},&&\Pi_2=\{1243,1324,1342\}, && \Pi_3&=\{1324,1342,1432\}\\
\Pi_4=\{2314,3214,4213\}, && \Pi_5=\{3214,3241,4213\}\, . && 
\end{align*}
Then, for all $j=1,2,3,4,5$,
\begin{equation}\label{maineqn}
\sum_{n\geq0}\#S_n(\Pi_j)x^n=\gfNEW\, .
\end{equation}

\end{theorem}

\section{Proof of main theorem}
To prove our main theorem, we find an explicit formula for the generating function $\sum_{n\geq0}\#S_n(\Pi_j)x^n$, where 
$j=1,2,3,4,5$. 
Furthermore, for the fifth class, $\Pi_5$, we give an explicit formula for the number of members of the set $S_n(\Pi_5)$.

\subsection{Class 1} $\Pi_1=\{1234,1243,1342\}$. Let $A_n=S_n(\Pi_1)$. Define $a_n=\#A_n$ and $a_n(i_1,\ldots,i_s)$ to 
be the number of permutations $\pi = \pi_1\cdots\pi_n\in A_n$ such that $\pi_1\cdots\pi_s=i_1\cdots i_s$. Then we have the 
following recurrence.

\begin{lemma}\label{Case1L1}
Define $b_n(i)=a_n(i,n-1)$. For all $1\leq i\leq n-3$,
\begin{align*}
a_n(i)&=a_{n-1}(i)+\cdots+a_{n-1}(1)+b_n(i),\\
b_n(i)&=b_{n-1}(i)+\cdots+b_{n-1}(1)
\end{align*}
with $a_n(n-2)=a_n(n-1)=a_n(n)=a_{n-1}$, $b_n(n-1)=0$ and $b_n(n-2)=b_n(n)=a_{n-2}$.
\end{lemma}
\begin{proof}
By the definitions, $a_n(n)=a_n(n-1)=a_n(n-2)=a_{n-1}$, $b_n(n-1)=0$ and $b_n(n-2)=b_n(n)=a_{n-2}$. If $1\leq i\leq n-2$, then
\begin{align*}
a_n(i)&=\sum_{j=1}^{i-1}a_n(i,j)+\sum_{j=i+1}^na_n(i,j)=\sum_{j=1}^{i-1}a_{n-1}(i)+a_n(i,n)+b_n(i)\\
&=\sum_{j=1}^{i}a_{n-1}(i)+b_n(i).
\end{align*}
Also,
\begin{align*}
b_n(i)&=\sum_{j=1}^{i-1}a_n(i,n-1,j)+\sum_{j=i+1}^{n-2}a_n(i,n-1,j)+a_n(i,n-1,n).
\end{align*}
By the definitions, $a_n(i,n-1,n)=0$ (the permutations in the question have subsequence $i,n-1,n,n-2$ which is order isomorphic to $1342$. Let $\pi=i(n-1)j\pi'\in A_n$ with $i+1\leq j\leq n-3$, since $\pi$ avoids $1234$ and $1342$, we see that $a_n(i,n-1,j)=0$. Clearly, $a_n(i,n-1,n-2)=a_{n-1}(i,n-2)=b_{n-1}(i)$. Thus,
\begin{align*}
b_n(i)&=\sum_{j=1}^{i-1}a_n(i,n-1,j)+b_{n-1}(i).
\end{align*}
Note that $\pi=i(n-1)j\pi'\in A_n$ with $1\leq j\leq i$ if and only if $j(n-2)\pi''\in A_{n-1}$, where $\pi''$ is a word obtained from $\pi'$ by decreasing each letter greater than $i$ by $1$. Hence, $a_n(i,n-1,j)=a_{n-1}(j,n-2)$, for all $j=1,2,\ldots,i-1$. In other words,
$b_n(i)=\sum_{j=1}^{i}b_{n-1}(j)$, as required.
\end{proof}

Define $A_n(v)=\sum_{i=1}^na_n(i)v^{i-1}$ and $B_n(v)=\sum_{i=1}^nb_n(i)v^{i-1}$. Then by multiplying the recurrence relations in Lemma \ref{Case1L1} by $v^{i-1}$, we obtain
\begin{align*}
\sum_{i=1}^{n-3}a_n(i)v^{i-1}&=\sum_{i=1}^{n-3}\sum_{j=1}^ia_{n-1}(j)v^{i-1}+\sum_{i=1}^{n-3}b_n(i)v^i,\\
\sum_{i=1}^{n-3}b_n(i)v^{i-1}&=\sum_{i=1}^{n-3}\sum_{j=1}^ib_{n-1}(j)v^{i-1},
\end{align*}
which, by the initial conditions, gives that for $n\geq3$,
\begin{align*}
A_n(v)&=\frac{1}{1-v}(A_{n-1}(v)-v^{n}A_{n-1}(1))+B_{n}(v)-v^{n-1}A_{n-2}(1),\\
B_n(v)&=\frac{1}{1-v}(B_{n-1}(v)-v^{n-3}B_{n-1}(1))+v^{n-3}A_{n-3}(1)+v^{n-1}A_{n-2}(1)+v^{n-3}A_{n-2}(1).
\end{align*}
By direct calculations, we have $A_0(v)=A_1(v)=1$, $A_2(v)=1+v$, $B_0(v)=B_1(v)=0$ and $B_2(v)=v$.

Let $A(x,v)=\sum_{n\geq0}A_n(v)x^n$ and $B(x,v)=\sum_{n\geq0}B_n(v)x^n$ be the generating functions for the sequences $A_n(v)$ and $B_n(v)$, respectively. By multiplying by $x^n$ and summing over $n\geq3$, we obtain
\begin{align}
&A(x,v)-1-x-(1+v)x^2\notag\\
&\quad=\frac{x}{1-v}(A(x,v)-1-x-vA(xv,1)+v+xv^2)+B(x,v)-vx^2A(xv,1),\label{eqcase1a1}\\
&B(x,v)-vx^2\notag\\
&\quad=\frac{x}{1-v}(B(x,v)-v^{-3}B(xv,1))+x^3A(xv,1)+(vx^2+x^2v^{-1})(A(xv,1)-1).\label{eqcase1a2}
\end{align}
Hence, \eqref{eqcase1a1} and \eqref{eqcase1a2} can be written as
\begin{align*}
\left(1-\frac{x}{v(1-v)}\right)A(x/v,v)&=1-\frac{x}{1-v}A(x,1)+B(x/v,v)-\frac{x^2}{v}A(x,1),\\
\left(1-\frac{x}{v(1-v)}\right)B(x/v,v)&=\frac{-x}{v^4(1-v)}B(x,1)+\left(\frac{x^3}{v^3}+\frac{x^2}{v}+\frac{x^2}{v^3}\right)A(x,1)-\frac{x^2}{v^3}.
\end{align*}
By substituting $v=\frac{1+\sqrt{1-4x}}{2}$ (the zero of the kernel $1-\frac{x}{v(1-v)}$, see \cite{HM}) into the second equation, we obtain
\begin{align}
B(x,1)=\frac{x(\sqrt{1-4x}-1)}{2}+\frac{2x^2+x-x\sqrt{1-4x}}{2}A(x,1).\label{eqcase1a3}
\end{align}
By multiplying the first equation by $1-\frac{x}{v(1-v)}$, and using the second equation, we obtain
\begin{align*}
\left(1-\frac{x}{v(1-v)}\right)^2A(x/v,v)&=1-\frac{x}{v(1-v)}-\frac{x^2}{v^3}-\left(\frac{x^2}{v}+\frac{x}{1-v}\right)\left(1-\frac{x}{v(1-v)}\right)A(x,1)\\
&-\frac{x}{v^4(1-v)}B(x,1)+\left(\frac{x^3}{v^3}+\frac{x^2}{v}+\frac{x^2}{v^3}\right)A(x,1).
\end{align*}
After differentiating the above equation respect to $v$, substituting $v=\frac{1+\sqrt{1-4x}}{2}$ together with using \eqref{eqcase1a3}, and several simple algebraic operations, we obtain an explicit formula for $A(x,1)$ as
$$A(x,1)=\gfNEW,$$
which completes the proof of this case.

\subsection{Class 2}
$\Pi_2=\{1243,1324,1342\}$. Let $A_n=S_n(\Pi_2)$. Define $a_n=\#A_n$ and $a_n(i_1,\ldots,i_s)$ to be the number of permutations $\pi_1\cdots\pi_n\in A_n$ such that $\pi_1\cdots\pi_s=i_1\cdots i_s$. 

\begin{lemma}\label{Case2L1}
Define $b_n(i)=a_n(i,i+1)$. For all $1\leq i\leq n-3$,
\begin{align*}
a_n(i)&=a_{n-1}(i)+\cdots+a_{n-1}(1)+b_n(i),\\
b_n(i)&=b_{n-1}(i)+\cdots+b_{n-1}(1)
\end{align*}
with $a_n(n-2)=a_n(n-1)=a_n(n)=a_{n-1}$, $b_n(n)=0$ and $b_n(n-2)=b_n(n-1)=a_{n-2}$.
\end{lemma}
\begin{proof}
By the definitions, $a_n(n)=a_n(n-1)=a_n(n-2)=a_{n-1}$, $b_n(n)=0$ and $b_n(n-2)=b_n(n-1)=a_{n-2}$. If $1\leq i\leq n-2$, then
\begin{align*}
a_n(i)&=\sum_{j=1}^{i-1}a_n(i,j)+\sum_{j=i+1}^na_n(i,j)=\sum_{j=1}^{i-1}a_{n-1}(i)+a_n(i,n)+b_n(i)\\
&=\sum_{j=1}^{i}a_{n-1}(i)+b_n(i).
\end{align*}
Also,
\begin{align*}
b_n(i)&=\sum_{j=1}^{i-1}a_n(i,i+1,j)+\sum_{j=i+2}^{n}a_n(i,i+1,j).
\end{align*}
By the definitions $a_n(i,i+1,j)=0$ with $j>i+2$ (the permutations in the question have subsequence $i,i+1,j,i+2$ which is order isomorphic to $1243$) and  $a_n(i,i+1,i+2)=a_{n-1}(i,i+1)=b_{n-1}(i)$. Thus
\begin{align*}
b_n(i)&=b_{n-1}(i)+\sum_{j=1}^{i-1}a_n(i,i+1,j).
\end{align*}
Let $\pi=i(i+1)j\pi'\in A_n$ with $1\leq j\leq i-1$. Then the letters $i,i+1,i+2,i+3,\ldots,n$ creates an increasing subsequence in $\pi$. If $j'$ with $j<j'<i$ appears on the right side of position of $i+2$ in $\pi$, then $\pi$ contains either $j(i+2)j'(i+3)$ or $j(i+2)(i+3)j'$ which is order isomorphic to $1324$ or $1342$, respectively. Thus $j'$ appears on the left side of the position of $i+2$ in $\pi$. Since $\pi$ avoids $1324$ then $\pi$ contains the subsequence $j,j+1,\ldots,i-1$. Thus $pi\in A_n$ if and only if $j (j+1) \pi''\in A_{n-1}$, where $\pi''$ is a word obtained from $\pi'$ by decreasing each letter greater than $i$ by $1$ and increasing the letters $j+1,j+2,\ldots,i-1$ by $1$. Hence, $a_n(i,i+1,j)=a_{n-1}(j,j+1)$, for all $j=1,2,\ldots,i-1$. In other words,
$b_n(i)=\sum_{j=1}^{i}b_{n-1}(j)$, as required.
\end{proof}

By using the techniques that have been used in the proof of Class 1 and the similarity of Lemma \ref{Case1L1} and Lemma \ref{Case2L1}, one can solve the recurrence relation in Lemma \ref{Case2L1}, and obtain that the generating function $A(x)=\sum_{n\geq0}a_nx^n$ is given by 
%(IF IS NEEDED, I CAN ADD MORE DETAILS)
$$\gfNEW,$$
as required.

\subsection{Class 3}
$\Pi_3=\{1324,1342,1432\}$. Let $A_n=S_n(\Pi_3)$. Define $a_n=\#A_n$ and $a_n(i_1,\ldots,i_s)$ to be the number of permutations $\pi_1\cdots\pi_n\in A_n$ such that $\pi_1\cdots\pi_s=i_1\cdots i_s$. By using similar arguments as in the proof of Lemmas \ref{Case1L1} and \ref{Case2L1}, one can state the following recurrence.

\begin{lemma}\label{Case3L1}
Define $b_n(i)=a_n(i,n)$. For all $1\leq i\leq n-3$,
\begin{align*}
a_n(i)&=a_{n-1}(i)+\cdots+a_{n-1}(1)+b_n(i),\\
b_n(i)&=b_{n-1}(i)+\cdots+b_{n-1}(1)
\end{align*}
with $a_n(n-2)=a_n(n-1)=a_n(n)=a_{n-1}$, $b_n(n)=0$ and $b_n(n-2)=b_n(n-1)=a_{n-2}$.
\end{lemma}

By comparing Lemma \ref{Case2L1} and Lemma \ref{Case3L1}, we obtain that $\#S_n(\Pi_2)=\#S_n(\Pi_3)$, which implies that the generating function $A(x)=\sum_{n\geq0}a_nx^n$ is given by
$$\gfNEW,$$
as required.

\subsection{Class 4}
$\Pi_4=\{2314,3214,4213\}$ We first give a bijection from permutations avoiding $\{3214,4213\}$ to
one-size-smaller Schr\"{o}der paths.

Recall that a \emph{Schr\"{o}der path} is a lattice path of North steps $N=(0,1)$, diagonal steps $D=(1,1)$ and East steps $E=(1,0)$
that starts at the origin, never drops below the diagonal $y=x$, and terminates on the diagonal.
Its \emph{size} is $\#\,N$ steps + $\#\,D$ steps, and a
Schr\"{o}der $n$-path is one of size $n$. Thus a Schr\"{o}der $n$-path ends at $(n,n)$.
The vertices on $y=x$ split a nonempty Schr\"{o}der path into its \emph{components}, and a Schr\"{o}der path whose
only vertices on $y=x$ are its endpoints (hence, is a one component path) is \emph{indecomposable}.
Thus all components of a Schr\"{o}der path are indecomposable.
The number of Schr\"{o}der $n$-paths is the large  Schr\"{o}der number $r_{n}$,
\htmladdnormallink{A006318}{http://oeis.org/A006318}.
A \emph{peak} is a pair of consecutive steps $NE$ (consider the path rotated $45^{\circ}$).
%A \emph{low} peak is one whose endpoints are on the diagonal $y=x$. Otherwise, it is a \emph{high} peak.

Every permutation on $[n]$ has a bounding up-down staircase (Figure 1) determined by its left to right (LR) maxima and right to left (RL) maxima.
\begin{center}
\begin{pspicture}(-5,-1)(5,6)%\showgrid
\psset{unit=.6cm}
\psline(-5,0)(5,0)\psline(-5,1)(5,1)\psline(-5,2)(5,2)\psline(-5,3)(5,3)\psline(-5,4)(5,4)\psline(-5,5)(5,5)
\psline(-5,6)(5,6)\psline(-5,7)(5,7)\psline(-5,8)(5,8)\psline(-5,9)(5,9)\psline(-5,10)(5,10)
%%%
\psline(-5,0)(-5,10)\psline(-4,0)(-4,10)\psline(-3,0)(-3,10)\psline(-2,0)(-2,10)
\psline(-1,0)(-1,10)\psline(0,0)(0,10)\psline(1,0)(1,10)\psline(2,0)(2,10)\psline(3,0)(3,10)
\psline(4,0)(4,10)\psline(5,0)(5,10)
\rput(-4.5,4.5){\textrm{{\footnotesize $5$}}}
\rput(-3.5,0.5){\textrm{{\footnotesize $1$}}}\rput(-2.5,3.5){\textrm{{\footnotesize $4$}}}
\rput(-1.5,8.5){\textrm{{\footnotesize $9$}}}\rput(-0.5,5.5){\textrm{{\footnotesize $6$}}}
\rput(0.5,7.5){\textrm{{\footnotesize $8$}}}\rput(1.5,9.5){\textrm{{\footnotesize $10$}}}
\rput(2.5,1.5){\textrm{{\footnotesize $2$}}}\rput(3.5,6.5){\textrm{{\footnotesize $7$}}}
\rput(4.5,2.5){\textrm{{\footnotesize $3$}}}
\psline[linewidth=2pt](-5,0)(-5,5)(-2,5)(-2,9)(1,9)(1,10)(2,10)(2,7)(4,7)(4,3)(5,3)(5,0)
%\psline[linecolor=gray](-6,0)(6,10)
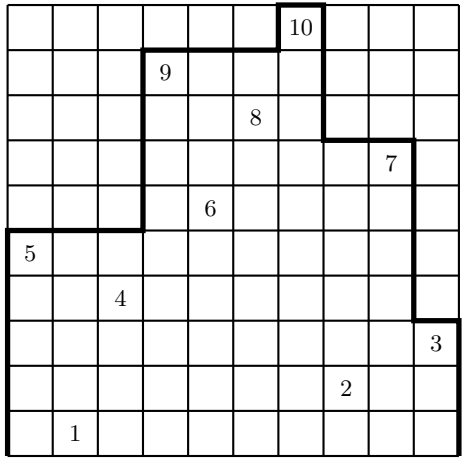
\rput(0,-1){Figure 1. Bounding staircase of a permutation.}
\end{pspicture}
\end{center}
\emph{Bounding staircases} of size $n$ are lattice paths from the origin consisting of $n$ steps each North $N=(0,1)$, East $E=(1,0)$, and South $S=(-1,0)$, that are characterized by the following properties:
\vspace*{-2mm}
\begin{enumerate}
\item all $N$ steps precede all $S$ steps, %\\[-8mm]
\item East runs (maximal sequence of contiguous $E$ steps) are at different heights,%\\[-8mm]
\item measuring from the top, the $i$-th pair of matching $N/S$ steps are at least $i$ units apart (to make room for the permutation entries above them), and the first pair are just 1 unit apart (they bracket the entry $n$)
\end{enumerate}

\begin{prop}\label{bij1} 
There is a bijection from bounding staircases to one-size-smaller Schr\"{o}der paths.  %\\[-10mm]
\end{prop}
%\begin{center}
\begin{pspicture}(-2.5,-2.5)(3,4.5)%\showgrid
\psset{unit=.4cm}
\psline[linecolor=gray](-5,0)(5,0)\psline[linecolor=gray](-5,1)(5,1)\psline[linecolor=gray](-5,2)(5,2)\psline[linecolor=gray](-5,3)(5,3)\psline[linecolor=gray](-5,4)(5,4)\psline[linecolor=gray](-5,5)(5,5)
\psline[linecolor=gray](-5,6)(5,6)\psline[linecolor=gray](-5,7)(5,7)\psline[linecolor=gray](-5,8)(5,8)\psline[linecolor=gray](-5,9)(5,9)\psline[linecolor=gray](-5,10)(5,10)
%%%
\psline[linecolor=gray](-5,0)(-5,10)\psline[linecolor=gray](-4,0)(-4,10)\psline[linecolor=gray](-3,0)(-3,10)\psline[linecolor=gray](-2,0)(-2,10)
\psline[linecolor=gray](-1,0)(-1,10)\psline[linecolor=gray](0,0)(0,10)\psline[linecolor=gray](1,0)(1,10)\psline[linecolor=gray](2,0)(2,10)\psline[linecolor=gray](3,0)(3,10)
\psline[linecolor=gray](4,0)(4,10)\psline[linecolor=gray](5,0)(5,10)
\psline[linewidth=2pt](-5,0)(-5,5)(-2,5)(-2,9)(1,9)(1,10)(2,10)(2,7)(4,7)(4,3)(5,3)(5,0)
\psdots[linewidth=.5pt,linecolor=white](4,7)(2,7)
\psdots[linewidth=.5pt,linecolor=white](4,3)(5,3)
\rput(6.5,5){$\rightarrow$}
\rput(0,-1){a)}
\end{pspicture}
\begin{pspicture}(-2,-2.5)(3,4.5)%\showgrid
\psset{unit=.4cm}
\psline[linecolor=gray](-5,0)(5,0)\psline[linecolor=gray](-5,1)(5,1)\psline[linecolor=gray](-5,2)(5,2)\psline[linecolor=gray](-5,3)(5,3)\psline[linecolor=gray](-5,4)(5,4)\psline[linecolor=gray](-5,5)(5,5)
\psline[linecolor=gray](-5,6)(5,6)\psline[linecolor=gray](-5,7)(5,7)\psline[linecolor=gray](-5,8)(5,8)\psline[linecolor=gray](-5,9)(5,9)\psline[linecolor=gray](-5,10)(5,10)
%%%
\psline[linecolor=gray](-5,0)(-5,10)\psline[linecolor=gray](-4,0)(-4,10)\psline[linecolor=gray](-3,0)(-3,10)\psline[linecolor=gray](-2,0)(-2,10)
\psline[linecolor=gray](-1,0)(-1,10)\psline[linecolor=gray](0,0)(0,10)\psline[linecolor=gray](1,0)(1,10)\psline[linecolor=gray](2,0)(2,10)\psline[linecolor=gray](3,0)(3,10)
\psline[linecolor=gray](4,0)(4,10)\psline[linecolor=gray](5,0)(5,10)
\psline[linewidth=2pt](-5,0)(-5,3)(-4,3)(-4,5)(-1,5)(-1,7)(1,7)(1,9)(4,9)(4,10)(5,10)(5,0)
\psdots[linewidth=2pt,linecolor=blue](-5,3)(-1,7)
\rput(6.5,5){$\rightarrow$}
\rput(0,-1){b)}
\rput(0,-2.8){Bijection from bounding staircases to one-size-smaller Schr\"{o}der paths}
\rput(0,-4.5){Figure 2.}
\end{pspicture}
\begin{pspicture}(-2,-2.5)(3,4.5)%\showgrid
\psset{unit=.4cm}
\psline[linecolor=gray](-5,0)(4,0)\psline[linecolor=gray](-5,1)(4,1)\psline[linecolor=gray](-5,2)(4,2)\psline[linecolor=gray](-5,3)(4,3)\psline[linecolor=gray](-5,4)(4,4)\psline[linecolor=gray](-5,5)(4,5)
\psline[linecolor=gray](-5,6)(4,6)\psline[linecolor=gray](-5,7)(4,7)\psline[linecolor=gray](-5,8)(4,8)\psline[linecolor=gray](-5,9)(4,9)
%%%
\psline[linecolor=gray](-5,0)(-5,9)\psline[linecolor=gray](-4,0)(-4,9)\psline[linecolor=gray](-3,0)(-3,9)\psline[linecolor=gray](-2,0)(-2,9)
\psline[linecolor=gray](-1,0)(-1,9)\psline[linecolor=gray](0,0)(0,9)\psline[linecolor=gray](1,0)(1,9)\psline[linecolor=gray](2,0)(2,9)\psline[linecolor=gray](3,0)(3,9)
\psline[linecolor=gray](4,0)(4,9)
\psline[linewidth=2pt](-5,0)(-5,2)(-4,3)(-4,5)(-1,5)(-1,6)(0,7)(1,7)(1,9)(4,9)
\psline[linecolor=gray](-5,0)(4,9)
\rput(0,-1){c)}
\end{pspicture}
\begin{proof} Delete each run of East steps bounded by two $S$ steps (Fig.\ 2a), insert it between the matching $N$ steps,
and color the newly introduced $NE$ corner blue (Fig.\ 2b). Then delete the last $n+2$ steps
(necessarily $N\,E\,S^n$) and replace each blue $NE$ corner with a diagonal step $D=(1,1)$ to get the
desired Schr\"{o}der path (Fig. 2c). \qed
\begin{lemma} \label{bij2} A permutation $p$ avoids $\{3214,4213\}$ if and only if it is lexicographically least among all
permutations with the same bounding staircase as $p$. %\\[-10mm]
\end{lemma}
\noindent Proof. If either offending pattern is present in $p$, then there is also a subsequence $xbay$ with $x$
a LR max, $y$ a RL max, $b>a$ and $x,y$ both $>b$. Switching the $a$ and $b$ gives a lexicographically smaller permutation with the
same LR max/RL max, both in value and position, and hence the same bounding staircase. Conversely, if $p$ is not lexicographically least,
then a $ba$ is present with $b>a$ and neither $a$ nor $b$
a LR max or RL max, implying that $ba$ is the ``21'' of an offending pattern. 
\end{proof}

\noindent \textbf{Remark.} To construct this lexicographically least permutation, use the bounding staircase to fill
the LR max and RL max slots in the permutation, then
fill the remaining slots right to left in turn with the largest available entry that will not create a new RL max.

\begin{cor} \label{bij3} 
The map ``permutation $\rightarrow$ bounding staircase'' is a bijection 
from \linebreak$S_n(3214,4213)$ to bounding staircases of size $n$.  %  \\[-8mm]
\end{cor}

Combining this bijection with that of Prop.\,\ref{bij1}, we have a bijection $\phi:S_n(3214,4213) \rightarrow $ Schr\"{o}der $(n-1)$-paths.

\begin{cor}\emph{\cite{kremer2000}} 
$\v S_n(3214,4213)\v = r_{n-1}$, the large Schr\"{o}der number.
\end{cor}

\begin{prop}\label{bij4} 
The restriction $\phi_{|S_n(\Pi_4)}$ is a bijection from $S_n(\Pi_4)$ to 
Schr\"{o}der $(n-1)$-paths in which each component has at most one peak. %\\[-10mm]
\end{prop}
\begin{proof}
In a 2314 pattern in a $\{3214,4213\}$-avoider $p$, the ``2'' and ``3'' must be LR maxima of $p$, and LR maxima in
the permutation correspond to peaks in the Schr\"{o}der path. Now consider the insertion of two dividers in $p$, one just before a LR max
and the other just after a RL max, to split $p$ into three segments $A,B,C$. Necessarily, $n \in B$ while $A,C$ may be empty.
Returns to $y=x$ in the Schr\"{o}der path correspond to such insertions for which $A \cup C$ is a nonempty initial segment of the positive integers. The shortest $AC$ thus corresponds to the first component of the Schr\"{o}der path. The ``2'' and ``3'' of the 2314 pattern either both lie in $A$ or both lie in $B$. If they lie in $A$, the ``1'' cannot lie in $B$. These observations are the basis for an inductive proof and allow us to assume that, in addition to $AC$ being shortest, $B$ is the singleton $n$, and so the Schr\"{o}der path has just one component.
If a 2314 is present, the ``2'' and ``3'' produce two peaks. On the other hand, if there are two peaks, they produce a ``2'' and ``3'',
and there must also be present a ``1'' and ``4'' to make a 2314 for otherwise $AC$ would not be shortest. 
\end{proof}

We have the following elementary counts for Schr\"{o}der paths.
\begin{lemma} \label{schrPeaks} For $n\ge 1$,\\
\emph{(i)} \emph{\cite[Ex.\,45]{stanleyCat}} The number of Schr\"{o}der $n$-paths with no peaks is the Catalan number $C_n$. \\
\emph{(ii)} \emph{[See \htmladdnormallink{A060693}{http://oeis.org/A060693}]} The number of  
Schr\"{o}der $n$-paths with exactly $1$ peak is $\binom{2n-1}{n-1}$.  %\\[-8mm]
\end{lemma}

An indecomposable Schr\"{o}der path of size $n\ge 2$ has the form $NPE$ with $P$ a Schr\"{o}der path of size $n-1$; hence we have

\begin{cor}  \label{indecCnts}
\emph{(i)} The number of indecomposable Schr\"{o}der $n$-paths with no peaks is
$2$ for $n=1$ and $C_{n-1}$ for $n\ge 2$. \\
\emph{(ii)} The number of indecomposable Schr\"{o}der $n$-paths with exactly $1$ peak is $0$ for $n=1$
and $\binom{2n-3}{n-2}$ for $n\ge 2$. %\\[-5mm]
\end{cor}

\begin{prop} \label{indecGF} The \gf for indecomposable Schr\"{o}der paths with at most $1$ peak is
\[
\frac{1}{2} \left(1+x+\frac{x}{\sqrt{1-4 x}}-\sqrt{1-4 x}\right)\, .
\]
\end{prop}
\begin{proof} Immediately by Corollary \ref{indecCnts}.
\end{proof}

\begin{cor} \label{gf1hiPeak}The \gf for Schr\"{o}der paths with at most $1$ peak in each component is
\[
\frac{2 \sqrt{1-4 x}}{1-5x+(1-x)\sqrt{1-4 x}}\, .
\]
\end{cor}
\begin{proof}
This \gf is the Invert transform of the \gf in Proposition \ref{indecGF}.
\end{proof}

\begin{cor}  The \gf for nonempty $\pi_4$-avoiders  is
\begin{equation} \label{corcase4}
\frac{2x \sqrt{1-4 x}}{1-5x+(1-x)\sqrt{1-4 x}}\, .
\end{equation}
\end{cor}
\begin{proof} 
Immediately by Proposition \ref{bij4} and Corollary \ref{gf1hiPeak}.
\end{proof}
Adding 1 to (\ref{corcase4}) to include the empty permutation gives (\ref{maineqn}).

\subsection{Class 5}
$\Pi_5=\{3214,3241,4213\}$.
To characterize $\Pi_5$-avoiders, draw a horizontal line just below the last entry of a permutation $p$ as in Figure 3 to obtain two subpermutations, $A$ above the line (in blue) and $B$ below the line (in black).
Split $A$ into two segments, $A_1$ consisting of the entries weakly left of $n$
and $A_2$ consisting of the remaining entries. Here, $A_1= (10,13,18),
\ A_2=(14,15,17,16,11,12,9)$. Say an entry in $p$ is \emph{key} if it either lies in $A_1$ or is a LR min in $A_2$ (key entries are circled in Figure 3 and we use the terms ``key'' and ``circled'' interchangeably below). Let $B_2$
denote the terminal segment of $B$ consisting of the entries that lie (in $p$) after the first entry of $A$. Here $B_2=(2,4,7,8)$.

\begin{center}
\begin{pspicture}(-1.8,-2)(12,11)%\showgrid
\psset{unit=.6cm}
\psline(0,0)(0,18)
\psline(1,0)(1,18)
\psline(2,0)(2,18)
\psline(3,0)(3,18)
\psline(4,0)(4,18)
\psline(5,0)(5,18)
\psline(6,0)(6,18)
\psline(7,0)(7,18)
\psline(8,0)(8,18)
\psline(9,0)(9,18)
\psline(10,0)(10,18)
\psline(11,0)(11,18)
\psline(12,0)(12,18)
\psline(13,0)(13,18)
\psline(14,0)(14,18)
\psline(15,0)(15,18)
\psline(16,0)(16,18)
\psline(17,0)(17,18)
\psline(18,0)(18,18)

\psline(0,0)(18,0)
\psline(0,1)(18,1)
\psline(0,2)(18,2)
\psline(0,3)(18,3)
\psline(0,4)(18,4)
\psline(0,5)(18,5)
\psline(0,6)(18,6)
\psline(0,7)(18,7)
\psline[linewidth=2pt](0,8)(18,8)
\psline(0,9)(18,9)
\psline(0,10)(18,10)
\psline(0,11)(18,11)
\psline(0,12)(18,12)
\psline(0,13)(18,13)
\psline(0,14)(18,14)
\psline(0,15)(18,15)
\psline(0,16)(18,16)
\psline(0,17)(18,17)
\psline(0,18)(18,18)

\rput(0.5,2.5){3}
\rput(1.5,4.5){5}
\rput(2.5,0.5){1}
\rput(3.5,5.5){6}
\rput(4.5,9.5){\circled{\blue{10}}}
\rput(5.5,1.5){2}
\rput(6.5,12.5){\circled{\blue{13}}}
\rput(7.5,17.5){\circled{\blue{18}}}
\rput(8.5,3.5){4}
\rput(9.5,6.5){7}
\rput(10.5,13.5){\circled{\blue{14}}}
\rput(11.5,14.5){\blue{15}}
\rput(12.5,16.5){\blue{17}}
\rput(13.5,15.5){\blue{16}}
\rput(14.5,7.5){8}
\rput(15.5,10.5){\circled{\blue{11}}}
\rput(16.5,11.5){\blue{12}}
\rput(17.5,8.5){\circled{\blue{9}}}

\rput(9,-1){A $\Pi_5$-avoider with $n=18$}
\rput(9,-2.3){Figure 3}
\end{pspicture}
\end{center}

Here are some properties of a $\Pi_5$-avoider $p=(p_1,\dots,p_n)$. Let $f$ and $l$ denote the first and last entries of $A$ respectively.
\vspace*{-1mm}
\begin{enumerate}
\item $A$, and hence St($A$), the standardization of $A$, is 213-avoiding, for if $bac$ is a 213 pattern in $A$, then each of $a,b,c$ is $>l$ and $bacl$ is a forbidden 3241 in $p$.
\item $B$ is 321-avoiding, for if $cba$ is a 321 pattern in $B$, then $cbal$ is a forbidden 3214 in $p$.
\item $B_2$ is increasing, for if $ba$ is a 21 in $B$ then $f \neq l$ and $fbal$ is either a 3214 or 4213 in $p$, both forbidden.
\item For every $x \in B$, the right neighbor $y$ of $x$ in $p$ (it always has one) is either also in $B$ or is circled, for otherwise $y$ is in $A_2$ but not a LR min of $A_2$, and so there is $z \in A_2$ lying to the left of both
$x$ and $y$ in $p$ with $z<y$. Then $nzxy$ is a forbidden 4213 in $p$.
\end{enumerate}
\vspace*{-1mm}
(Note that item 4 says that if $B$ is divided into blocks of entries that are contiguous in $p$, then each block lies immediately to the left of a circled entry in $p$.) Conversely, if these 4 conditions are met, the reader may check that $p$ is a $\Pi_5$-avoider.

Now, to count $\Pi_5$-avoiders, we first dispose of the cases where $A$ has length 1,\,2 or $n$.
\begin{lemma} \label{first3}
Suppose $n\ge 3$. Then for each of $a=1,2$ and $n$, we have $\v \{p \in S_n(\Pi_5): \textrm{\emph{length(}}A\emph{)}=a\}\v = C_{n-1}$.
\end{lemma}
\begin{proof}
Recall that both 321-avoiders and 213-avoiders on $[n]$ are counted by $C_n$
We have $a=1$ if and only of $n$ is the last entry of $p$. Avoidance of 3214 then implies $p\backslash\{n\}$ avoids $321$. Conversely, if $p\backslash\{n\}$
avoids 321 then, a fortiori, $p\backslash\{n\}$ avoids $\Pi_5$ and so does $p$.
Next, $a=2$ if and only of $n-1$ is the last entry of $p$.
Suppose $n-1$ is the last entry of $p$ and $p$ is a  $\Pi_5$-avoider. If $cba$ were a 321 pattern in $p$,
then $cba\,(n-1)$ would be a 4213 if $c=n$ and a 3214 if $c<n-1$, both of which are forbidden. So $p$ must avoid
321. Conversely, if $p\backslash\{n-1\}$
avoids 321 then, again, $p$ avoids $\Pi_5$.
Lastly, $a=n$ if and only of $1$ is the last entry of $p$ and then $p$ is a $\Pi_5$-avoider if and only of $p$ avoids 213 (else a 3241
terminating at the last entry is present) and the result follows.  
\end{proof}

For the remaining cases, we have  $3\le a \le n-1$ and so $n\ge 4$.
Then $k \ge 3$ as follows. Since $p_n \le n-2$ by the proof of
Lemma \ref{first3}, the three entries $n$, the successor of $n$ in A, and $p_n$ are all key and all distinct
unless $n$ is the second to last entry of $A$, but in that case $n-1$ occurs before $n$ and so is a key entry,
and $p_n,n-1,n$ are distinct. So $3 \le k \le a$.

The following elementary counting results will be useful; we omit the proofs. We use $C_{n,k}$ for 
the  generalized Catalan number $\frac{k+1}{2n+k+1}\binom{2n+k+1}{n}$. 
Recall that $(C_{n,k})_{n\ge 0}$ is the 
$(k+1)$-fold convolution of the Catalan numbers $(C_n)_{n\ge 0}=(C_{n,0})_{n\ge 0}$ and so the 
\gf $\sum_{n\ge 0}C_{n,k}x^n$ is given by $C(x)^{k+1}$ where $C(x):= \frac{1-\sqrt{1-4x}}{2x}$ is the \gf for the Catalan numbers. It is convenient below to use the convention $C_{0,-1}:=1$.

\begin{prop}\white{.} \newline
$($i$\,)$\ The number of \emph{213}-avoiding permutations on $[n]$ whose last entry is $1$ with $n$ in
first position and $k$ key entries is $C_{n-k,k-3}$ for $2 \le k \le n$. \newline
$($ii$\,)$\ The number of \emph{213}-avoiding permutations on $[n]$ whose last entry is $1$ with $n$ in
position $j$ and $k$ key entries is $C_{n-k,k-2-j}$ for $1 \le j \le k-1,\ k \le n$.
\end{prop}

\begin{cor}\label{cnt213}
The number of \emph{213}-avoiding permutations on $[n]$ whose last entry is $1$ with $k$ key entries is
$w(n,k) := \sum_{j=1}^{n-1}\binom{k-2}{j-1}C_{n-k,k-2-j}$ for $n\ge 2,\ 1 \le k \le n$.
\end{cor}

\begin{lemma}\label{cnt321}
The number of \emph{321}-avoiding permutations on $[n]$ in which the last $i$ entries are increasing is
$C_{n-i,i}$ for $0\le i \le n$.
\end{lemma}

We are now ready to count permutations $p$ in $S_n(\Pi_5)$ by $a:=$ length($A$), $k:=$
number of key entries, $i:=$ number
of entries of $B$ after the first circled entry in $p$. 
The cases $a=1,2$ or $n$ have been treated already.
So suppose given $n,a,k,i,$ with $3 \le k \le a \le n-1$ and  $0 \le i \le b:=n - a$.
By  Cor. \ref{cnt213}, there are $w(a,k)$ 213-avoiding permutations $A_1$ of length $a$
that end with 1 and have $k$ key entries.
By Lemma \ref{cnt321}, there are $C_{b-i,i}$ 321-avoiding permutations of length $b$ such that the last $i$ entries 
are increasing. There are $\binom{i+k-2}{i}$ ways to distribute these last $i$ entries into $k-1$ blocks to be placed 
just before the $k-1$ non-first key entries of $A=A_1+b$. (Of course, the initial block of $b-i$ entries of $B$ lies before the first key entry.)
These choices uniquely determine a $\Pi_5$-avoider of length $n$.

Hence, summing over $a,k,i$, we have for $n\ge 3$,
\begin{eqnarray}\label{gensum}
\v S_n(\Pi_5) \v & = & 3 C_{n-1} + \sum_{a=3}^{n-1}\sum_{k=3}^{a}\sum_{i=0}^{b}w(a,k)C_{b-i,i}\binom{i+k-2}{i} \\
& = & 3 C_{n-1} + \sum_{a=3}^{n-1}\sum_{k=3}^{a}\sum_{i=0}^{b}\sum_{j=1}^{a-1}\binom{k-2}{j-1}
C_{a-k,k-j-2}\, C_{b-i,i}\binom{i+k-2}{i} \nonumber \\
& = &  3 C_{n-1} + \sum_{a=3}^{n-1}\sum_{k=3}^{a}\sum_{j=1}^{a-1}\binom{k-2}{j-1}C_{a-k,k-j-2}\,C_{n-a,k-1}\, . \nonumber
\end{eqnarray}
The last equality evaluates the sum over $i$ using a generalized Catalan number identity.
The \gf $F(x):=\sum_{n\ge 0}\v S_{n}(\Pi_5)\v x^n$ is easily deduced:
\[
F(x) = 1+x+2x^2 + 3 \sum_{n\ge 3}C_{n-1}x^n + G(x)\, ,
\]
where 
\begin{eqnarray*}
G(x) & = & \sum_{n\ge 4}\sum_{a=3}^{n-1}\sum_{k=3}^{a}\sum_{j=1}^{a-1}\binom{k-2}{j-1}C_{a-k,k-j-2}C_{n-a,k-1}x^n \\
     & = & \sum_{k\ge 3}\sum_{j=1}^{k-1}\binom{k-2}{j-1}\sum_{a \ge k}C_{a-k,k-j-2}\sum_{n\ge a+1}C_{n-a,k-1}x^n \\
     & = & \sum_{k\ge 3}\big(C(x)^k-1\big)\sum_{j=1}^{k-1}\binom{k-2}{j-1}\sum_{a \ge k}C_{a-k,k-j-2}  x^a \\
     & = & \sum_{k\ge 3}x^k \big(C(x)^k-1\big)\sum_{j=1}^{k-1}\binom{k-2}{j-1} C(x)^{k-j-1} \\
     & = & \sum_{k\ge 3}x^k \big(C(x)^k-1\big)\big(1+C(x)\big)^{k-2} \, ,
\end{eqnarray*}
which is a difference of geometric sums. After evaluation and simplification, we find
\[
F(x) = 1 + \frac{2x \sqrt{1-4 x}}{1-5x+(1-x)\sqrt{1-4 x}}\, ,
\]
agreeing with the expression in (\ref{maineqn}), or with rationalized denominator,
\[
F(x) = 1 +\frac{2 x^2 + x(1 - 5 x) C(x)}{1 - 4 x - x^2}\, .
\]
In conclusion, we remark that the above characterization of $\Pi_5$-avoiders can easily be adapted to find the bivariate \gf for $\Pi_5$-avoiders by length and number of components. First, we count 
indecomposable $\Pi_5$-avoiders. For $n\ge 4$, the cases $a=1,2,n$ are counted by $0,C_{n-2},C_{n-1}$ respectively.
For $3\le a \le n-1$,  a $\Pi_5$-avoider is indecomposable iff $B$, in the notation above, in addition to being a 
321-avoider whose last $i$ entries are increasing, satisfies the property that for all $r=1,2,\dots,b-i$, the 
first $r$ entries of $B$, when sorted, do not form an initial segment of the positive integers (the property is vacuously 
satisfied when $i=b$). The number of such permutations is $C_{b - i, i - 1} = C_{n-a-i,i-1}$. 
Thus, in (\ref{gensum}), the initial $3 C_{n-1}$ term is replaced by $C_{n-2}+C_{n-1}$ and the $C_{b-i,i}$ factor in the 
sum is replaced by $C_{b - i, i - 1}$. This modified sum leads to the counting sequence $(1, 1, 3, 11, 43, 173, 707, \dots)_{n\ge 1}$, 
\htmladdnormallink{A026671}{http://oeis.org/A026671}, for indecomposable $\Pi_5$-avoiders, with 
\gf $F_{\textrm{indec}}(x):=1/(1-x/\sqrt{1-4x})$.
Further, a $\Pi_5$-avoider with $k\ge 2$ components has the form $p_1 \oplus \dots \oplus p_{k-1} \oplus p_k$ where
$p_1,\dots,p_{k-1}$ are all indecomposable 321-avoiders and $p_k$ is an indecomposable $\Pi_5$-avoider. 
Here $\oplus$ is the direct sum defined on permutations $\pi$ of length $m$ and $\sigma$ of length $n$ by
\[
\begin{array}{lcll} 
  (\pi\oplus\sigma)(i) &=&
  \left\{ 
    \begin{array}{l}
      \pi(i) \\
      \sigma(i-m)+m
    \end{array}
  \right.  &
  \begin{array}{l}
    \mbox{if $1\le i\le m$,} \\
    \mbox{if $m+1\le i \le m+n$.}
  \end{array} \\[15pt] 
  \end{array} 
\]
Since indecomposable 321-avoiders have the \gf $xC(x)$, the desired bivariate \gf, excluding the empty permutation, is 
\[
\frac{F_{\textrm{indec}}(x)y}{1 - x y C(x)} =  \frac{2  x y \sqrt{1-4 x}}{y-2 x -3 x y + (2-x y-y) \sqrt{1-4 x} }\, .
\]

%---------------------------

\end{document}